\documentclass{amsart} 
\usepackage{amsthm}
\newtheorem{theorem}{Theorem}[section]
\newtheorem{corollary}[theorem]{Corollary}
\newtheorem{lemma}[theorem]{Lemma}

\theoremstyle{definition}
\newtheorem{definition}[theorem]{Definition}
\theoremstyle{remark}
\newtheorem{remark}[theorem]{Remark}

\newcommand{\eps}{\varepsilon}
\newcommand{\E}{\mathsf E}
\newcommand{\Prob}{\mathsf P}
\newcommand{\R}{\mathbb R}
\newcommand{\F}{\mathcal F}
\newcommand{\wt}{\widetilde}
\newcommand{\wh}{\widehat}
\DeclareMathOperator{\const}{const}
\newcommand*{\abs}[1]{\left\vert#1\right\vert}
\newcommand*{\norm}[1]{\left\Vert#1\right\Vert}
\newcommand*{\set}[1]{\left\{#1\right\}}

\begin{document}
\title{Smooth approximations for fractional and multifractional fields}

\author{Kostiantyn Ralchenko}
\address{Department of Probability, Statistics and Actuarial Mathematics,
Mechanics and Mathematics Faculty,
Taras Shevchenko National University of Kyiv,
64, Volodymyrs'ka St.,
01601 Kyiv, Ukraine}
\email{k.ralchenko@gmail.com}

\author{Georgiy Shevchenko}
\email{zhora@univ.kiev.ua}

\begin{abstract}
We construct absolute continuous stochastic processes that converge to anisotropic fractional and multifractional Brownian sheets in Besov-type spaces.\end{abstract}



\thanks{Supported in part by the Commission of the European Communities grant PIRSES-GA-2008-230804 within the programme ``Marie Curie Actions''.}

\maketitle
\section{Introduction}
\emph{Fractional Brownian motion} (fBm) with Hurst parameter $H$ is a continuous centered Gaussian process with the covariance function
\[
\E B_tB_s=\frac12\left(t^{2H}+s^{2H}-\abs{t-s}^{2H}\right),
\quad t,s\ge0.
\]
It has stationary  increments that exhibit a property of long-range dependence for $H>1/2$, which makes fBm a popular and efficient model for long-range dependent processes in Internet traffic, stock markets, etc.

For numerous applications one needs a multi-parameter model for a long range dependence.
For most of those applications, including image processing,
geophysics, oceanography, etc,  two parameters are enough.

There are different possibilities to define a two-parameter fractional Brownian motion, or \emph{fractional Brownian sheet} (fBs).
One is so-called \emph{isotropic} fractional Brownian sheet with the covariance function
\[
\E B_tB_s=\frac12\left(\norm t^{2H}+\norm s^{2H}-\norm{t-s}^{2H}\right),
\quad t,s\in\R^2_+.
\]
As the name suggests, its properties are the same in all directions, which is not the case in many applications, especially in image processing.
For such applications a better model is \emph{anisotropic} fractional Brownian sheet with the covariance function
\[
\E B_tB_s=\frac14\prod\limits_{i=1,2}
    \left(t_i^{2H_i}+s_i^{2H_i}-\abs{t_i-s_i}^{2H_i}\right),
\quad t,s\in\R^2_+.
\]
It also has stationary increments in the sense that
the distribution of $(B_{t+s}-B_t)$ does not depend on $t$.

But the stationarity of increments of fractional Brownian process and sheet means that the behavior of them is the same
at each point, and this substantially restricts the area of their application.
In particular, they does not allow one to model situations,
where the regularity at a point depends on the point, as well as the range of dependence.
In view of this, recently several \emph{multifractional} generalizations of fBm and fBs were proposed in order to overcome these limitations such as moving average multifractional Brownian motion (mBm)~\cite{PeltierLevyVehel}, harmonisable mBm~\cite{BenassiJaffardRoux97}, Volterra-type mBm~\cite{mBm-en},
Different multiparameter extensions of mBm were studied in~\cite{AyacheLeger,Herbin,MeerschaertWuXiao}.

In this paper, we work in a general setting by considering a Gaussian random field $B(t)$ in the plane which is continuous almost surely and satisfies the following condition on its increments: for all $s,t\in[0,T_1]\times[0,T_2]$
\[
\E(B(s_1,s_2)-B(s_1,t_2)-B(t_1,s_2)+B(t_1,t_2))^2\le C(\abs{t_1-s_1}\abs{t_2-s_2})^\lambda,
\]
where $C>0$ and $\lambda>1$ are some constants. This, in particular, includes anisotropic fractional and multifractional
Brownian sheets.

Our main goal is to construct approximations for $B(t)$ in a certain Besov, or fractional Sobolev, spaces, by absolutely continuous fields.
This allows one to approximate stochastic integrals with respect to fractional Brownian sheet by usual integrals and consequently,
to approximate solutions of stochastic partial differential equations involving fractional noise by solutions of partial differential equations
with a random source, which in many aspects are similar to non-random partial differential equations.

The paper is organized as follows.
In Section~\ref{sec2} we give necessary definitions concerning Besov spaces. In Section~\ref{sec3} the main result, Theorem~\ref{th2.1}, about
the convergence of absolutely continuous approximations of the field $B(t)$ in Besov space $W^1$ in probability, is proved.
In Section~\ref{sec4} we study fractional and multifractional Brownian sheets. We show in particular that these fields satisfy the
conditions of Theorem~\ref{th2.1} and can be approximated by absolutely continuous random fields.

\section{Definitions and notation}
\label{sec2}
In this section, we define functional spaces which are similar to spaces of H\"older continuous functions.
They play an important role in definition and analysis of stochastic integrals with respect to fractional random fields (see e.\,g. \cite[Lemma 2.2.16]{Mishura08}).

Let $s,t\in\R^2_+$, $s=(s_1,s_2)$, $t=(t_1,t_2)$.
Write $s<t$ if $s_1<t_1$ and $s_2<t_2$.
For $s<t$ denote
$[s,t]=[s_1,t_1]\times[s_2,t_2]\subset\R^2_+$.
Let
$T=(T_1,T_2)\in(0,\infty)^2$,
$[0,T]=[0,T_1]\times[0,T_2]$.
For a function
$f\colon\R^2_+\to\R$
we consider two-parameter increments
\[
\Delta_sf(t):=f(t)-f(s_1,t_2)-f(t_1,s_2)+f(s), \quad s,t\in\R_+^2.
\]

Consider a function
$f\colon[0,T]\to\R$.
For
$\beta\in(0,1)^2$ and
$t\in[0,T]$
we denote
\begin{align*}
\varphi_1^{\beta_1}(f)(t)&=
    \int_0^{t_1}\frac{\abs{f(t)-f(s_1,t_2)}}{(t_1-s_1)^{\beta_1+1}}\,ds_1,\\
\varphi_2^{\beta_2}(f)(t)&=
    \int_0^{t_2}\frac{\abs{f(t)-f(t_1,s_2)}}{(t_2-s_2)^{\beta_2+1}}\,ds_2,\\
\varphi_3^{\beta_1\beta_2}(f)(t)&=
    \int_{[0,t]}\frac{\abs{\Delta_sf(t)}}{(t_1-s_1)^{\beta_1+1}(t_2-s_2)^{\beta_2+1}}\,ds,\\
\varphi_f^{\beta_1\beta_2}(t)&=
    \abs{f(t)}+\varphi_1^{\beta_1}(f)(t)+\varphi_2^{\beta_2}(f)(t)
    +\varphi_3^{\beta_1\beta_2}(f)(t).
\end{align*}
Let
$W_0^{\beta_1,\beta_2} = W_0^{\beta_1,\beta_2}([0,T])$
be a space of measurable functions
$f\colon[0,T]\to\R$,
with
\[
\norm{f}_{0,\beta_1,\beta_2}
=\sup_{t\in[0,T]}\varphi_f^{\beta_1\beta_2}(t)<\infty,
\]
Define also
$W_1^{\beta_1,\beta_2} = W_1^{\beta_1,\beta_2}([0,T])$
as a space of measurable functions
$f\colon[0,T]\to\R$,
with
\begin{align*}
\norm{f}_{1,\beta_1,\beta_2}
&=\sup_{0\le s<t\le T}
    \left(\frac{\abs{\Delta_sf(t)}}{(t_1-s_1)^{\beta_1}(t_2-s_2)^{\beta_2}}\right.\\
&\quad+\frac{1}{(t_2-s_2)^{\beta_2}}\int_{s_1}^{t_1}
    \frac{\abs{f_{t-}(u,s_2)-f_{t-}(s)}}{(u-s_1)^{1+\beta_1}}\,du\\
&\quad+\frac{1}{(t_1-s_1)^{\beta_1}}\int_{s_2}^{t_2}
    \frac{\abs{f_{t-}(s_1,v)-f_{t-}(s)}}{(v-s_2)^{1+\beta_2}}\,dv\\
&\quad\left.+\int_{[s,t]}\frac{\abs{\Delta_sf(r)}}
    {(r_1-s_1)^{1+\beta_1}(r_2-s_2)^{1+\beta_2}}\,dr\right)<\infty,
\end{align*}
where
$f_{t-}(s):=f(s)-f(s_1,t_2-)-f(t_1-,s_2)+f(t-)$.

\section{Main result}
\label{sec3}
Let $\set{B_t,t\in[0,T]}$ be a random field which satisfies the following conditions
\begin{enumerate}
    \item $B_t$ is a Gaussian field;
    \item there exists constants $C>0$ and $\lambda>1$ such that for all $s,t\in[0,T]$
    \begin{equation}\label{0}
    \E(\Delta_sB_t)^2\le C(\abs{t_1-s_1}\abs{t_2-s_2})^\lambda;
    \end{equation}
    \item the trajectories of $B_t$ are continuous with probability one.
\end{enumerate}

One example of such field is anisotropic fractional Brownian sheet (see Introduction).

We consider the approximation
\[
B_t^\eps=\frac{1}{\eps^2}\int_{t_1}^{t_1+\eps}\int_{t_2}^{t_2+\eps}B_s\,ds
=\frac{1}{\eps^2}\int_{[0,\eps]^2}B_{s+t}\,ds.
\]

\begin{theorem}\label{th2.1}
For all
$\beta_1,\beta_2\in(0,\lambda/2)$
\[
\norm{B^\eps-B}_{1,\beta_1,\beta_2}\xrightarrow{\Prob}0, \quad \eps\to0+.
\]
\end{theorem}

\begin{proof}
Denote
\[
\Delta B_t^\eps:=B_t^\eps-B_t
=\frac{1}{\eps^2}\int_{[0,\eps]^2}\left(B_{u+t}-B_t\right)\,du.
\]
Then
\begin{multline*}
\Delta_s(\Delta B_t^\eps)
=\Delta B_t^\eps-\Delta B_{s_1,t_2}^\eps
    -\Delta B_{t_1,s_2}^\eps+\Delta B_s^\eps
=\frac{1}{\eps^2}\int_{[0,\eps]^2}\left(B_{u+t}-B_t\right.\\
\left.-B_{u_1+s_1,u_2+t_2}+B_{s_1,t_2}-B_{u_1+t_1,u_2+s_2}+B_{t_1,s_2}+B_{u+s}-B_s\right)\,du.
\end{multline*}
According to the Cauchy--Schwarz inequality,
\begin{multline}\label{1}
\E(\Delta_s(\Delta B_t^\eps))^2
\le\frac{1}{\eps^2}\int_{[0,\eps]^2}\E\left(B_{u+t}-B_t-B_{u_1+s_1,u_2+t_2}+B_{s_1,t_2}\right.\\ \left.-B_{u_1+t_1,u_2+s_2}+B_{t_1,s_2}+B_{u+s}-B_s\right)^2\,du.
\end{multline}
Considering~\eqref{0}, we obtain
\begin{equation}\label{2}
\begin{split}
&\E(\Delta_s(\Delta B_t^\eps))^2
\le\frac{2}{\eps^2}\int_{[0,\eps]^2}\left(\E(\Delta_{u+s}B_{u+t})^2+\E(\Delta_sB_t)^2\right)\,du\\
&\quad\le\frac{2}{\eps^2}\int_{[0,\eps]^2}2C(\abs{t_1-s_1}\abs{t_2-s_2})^\lambda\,du
=4C(\abs{t_1-s_1}\abs{t_2-s_2})^\lambda.
\end{split}
\end{equation}
On the other hand, \eqref{1} implies
\begin{equation}\label{3}
\begin{split}
\E(\Delta_s(\Delta B_t^\eps))^2
&\le\frac{4}{\eps^2}\int_{[0,\eps]^2}\left(\E(\Delta_{u_1+s_1,t_2}B_{u+t})^2
+\E(\Delta_{t_1,s_2}B_{u_1+t_1,t_2})^2\right.\\
&\quad\left.+\E(\Delta_{u_1+s_1,s_2}B_{u_1+t_1,u_2+s_2})^2
+\E(\Delta_sB_{u_1+s_1,t_2})^2\right)\,du\\
&\le\frac{8C}{\eps^2}\int_{[0,\eps]^2}\left(\abs{t_1-s_1}^\lambda u_2^\lambda+u_1^\lambda\abs{t_2-s_2}^\lambda\right)\,du\\
&=\frac{8C}{\lambda+1}\left(\abs{t_1-s_1}^\lambda+\abs{t_2-s_2}^\lambda\right)\eps^\lambda,
\end{split}
\end{equation}
because, in view of~\eqref{0},
\begin{gather*}
\E(\Delta_{u_1+s_1,t_2}B_{u+t})^2\le C\abs{t_1-s_1}^\lambda u_2^\lambda,\\
\E(\Delta_{t_1,s_2}B_{u_1+t_1,t_2})^2\le Cu_1^\lambda\abs{t_2-s_2}^\lambda,\\
\E(\Delta_{u_1+s_1,s_2}B_{u_1+t_1,u_2+s_2})^2\le C\abs{t_1-s_1}^\lambda u_2^\lambda,\\
\E(\Delta_sB_{u_1+s_1,t_2})^2\le Cu_1^\lambda\abs{t_2-s_2}^\lambda.
\end{gather*}

Let
$\delta\in(0,\lambda-2\max\set{\beta_1,\beta_2})$.
We study three cases.

\emph{Case 1:}
$\abs{t_1-s_1}<\eps$.
Based on the estimate~\eqref{2}, we get
\[
\E(\Delta_s(\Delta B_t^\eps))^2
\le4C(\abs{t_1-s_1}\abs{t_2-s_2})^\lambda
<4CT_2^\delta(\abs{t_1-s_1}\abs{t_2-s_2})^{\lambda-\delta}\eps^{\delta}.
\]

\emph{Case 2:}
$\abs{t_1-s_1}\ge\eps$, $\abs{t_2-s_2}<\eps$.
Similarly to the Case~1, we obtain
\[
\E(\Delta_s(\Delta B_t^\eps))^2
\le4C(\abs{t_1-s_1}\abs{t_2-s_2})^\lambda
<4CT_1^\delta(\abs{t_1-s_1}\abs{t_2-s_2})^{\lambda-\delta}\eps^{\delta}.
\]

\emph{Case 3:}
$\abs{t_1-s_1}\ge\eps$, $\abs{t_2-s_2}\ge\eps$.
Based on the estimate~\eqref{3}, we get
\begin{align*}
\E(\Delta_s(\Delta B_t^\eps))^2
&\le\frac{8C}{\lambda+1}\left(\abs{t_1-s_1}^\lambda+\abs{t_2-s_2}^\lambda\right)\eps^\lambda\\
&<4C\left(T_1^\delta+T_2^\delta\right)(\abs{t_1-s_1}\abs{t_2-s_2})^{\lambda-\delta}\eps^{\delta}.
\end{align*}

Thus, in all 3 cases we have
\[
\E(\Delta_s(\Delta B_t^\eps))^2
<4C\left(T_1^\delta+T_2^\delta\right)(\abs{t_1-s_1}\abs{t_2-s_2})^{\lambda-\delta}\eps^{\delta}.
\]
Since
$\Delta_s(\Delta B_t^\eps)$
has a normal distribution, then for $p>0$ we have
\begin{equation}\label{4}
\E\abs{\Delta_s(\Delta B_t^\eps)}^p
\le C_1(\abs{t_1-s_1}\abs{t_2-s_2})^{(\lambda-\delta)p/2}\eps^{\delta p/2},
\end{equation}
where
$C_1=C_1(p,T_1,T_2)=\frac{2^{3p/2}C^{1/2}(T_1^\delta+T_2^\delta)^{1/2}\Gamma(\frac{p+1}{2})}{\Gamma(\frac12)}$.

According to the two-parameter Garsia--Rodemich--Rumsey inequality (\cite[Theorem 2.1]{GRRinequality}), for all
$p>0$, $\alpha_1>p^{-1}$, $\alpha_2>p^{-1}$
there exists a constant
$C_2=C_2(\alpha_1,\alpha_2,p)>0$
such that
\begin{equation}\label{5}
\abs{\Delta_s(\Delta B_t^\eps)}^p
\le C_2\abs{t_1-s_1}^{\alpha_1p-1}
\abs{t_2-s_2}^{\alpha_2p-1}\xi,
\end{equation}
where
\[
\xi=\int_{[0,T]^2}\frac{\abs{\Delta_x(\Delta B_y^\eps)}^p}%
{\abs{x_1-y_1}^{\alpha_1p+1}\abs{x_2-y_2}^{\alpha_2p+1}}\,dx\,dy.
\]

We choose
$0<\theta<(\lambda-2\max\set{\beta_1,\beta_2}-\delta)/2$,
$p=\frac2\theta$,
$\alpha_1=\alpha_2=\frac{\lambda-\theta-\delta}{2}$.
Then, taking into account~\eqref{4}, we obtain
\begin{align*}
\E\xi&=\int_{[0,T]^2}\frac{\E\abs{\Delta_x(\Delta B_y^\eps)}^p}{\abs{x_1-y_1}^{\alpha_1p+1}\abs{x_2-y_2}^{\alpha_2p+1}}\,dx\,dy\\
&\le C_1\eps^{\delta p/2}\int_{[0,T]^2}\abs{x_1-y_1}^{\frac{\lambda-\delta}{2}p-(\alpha_1p+1)}
\abs{x_2-y_2}^{\frac{\lambda-\delta}{2}p-(\alpha_2p+1)}\,dx\,dy\\
&=C_1\eps^{\delta p/2}\int_{[0,T]^2}dx\,dy
=C_1T_1^2T_2^2\eps^{\delta p/2}.
\end{align*}
Therefore, from~\eqref{5} we get
\begin{equation}\label{6}
\E\sup_{s,t\in[0,T]}\frac{\abs{\Delta_s(\Delta B_t^\eps)}^p}%
{(\abs{t_1-s_1}\abs{t_2-s_2})^{p(\lambda-2\theta-\delta)/2}}
\le C_3\eps^{\delta p/2},
\end{equation}
where $C_3=C_1C_2T_1^2T_2^2$.

The last estimate implies that for any
$\kappa\in(0,1)$
there exists
$c_\kappa$
such that probability of the event
\[
A_{\eps}:=\set{\text{for all $s,t\in[0,T]$}: \abs{\Delta_s(\Delta B_t^\eps)}\le
c_\kappa(\abs{t_1-s_1}\abs{t_2-s_2})^{(\lambda-2\theta-\delta)/2}\eps^{\delta/2}}
\]
is not less than $1-\kappa$.

As we have chosen
$0<\delta<\lambda-2\max\set{\beta_1,\beta_2}$ and
$0<\theta<\frac{\lambda-\delta}{2}-\max\set{\beta_1,\beta_2}$,
then
$h_1:=\frac{\lambda-\delta}{2}-\beta_1-\theta>0$,
$h_2:=\frac{\lambda-\delta}{2}-\beta_2-\theta>0$.

By the definition of the norm
$\norm{\cdot}_{1,\beta_1,\beta_2}$,
we have
\begin{align*}
&\norm{\Delta B^\eps}_{1,\beta_1,\beta_2}
=\sup_{0\le s<t\le T}
    \left(\frac{\abs{\Delta_s(\Delta B_t^\eps)}}{(t_1-s_1)^{\beta_1}(t_2-s_2)^{\beta_2}}\right.\\
&\quad+\frac{1}{(t_2-s_2)^{\beta_2}}\int_{s_1}^{t_1}
    \frac{\abs{\Delta_s(\Delta B_{u,t_2}^\eps)}}{(u-s_1)^{1+\beta_1}}\,du
+\frac{1}{(t_1-s_1)^{\beta_1}}\int_{s_2}^{t_2}
    \frac{\abs{\Delta_s(\Delta B_{t_1,v}^\eps)}}{(v-s_2)^{1+\beta_2}}\,dv\\
&\quad\left.+\int_{[s,t]}\frac{\abs{\Delta_s(\Delta B_r^\eps)}}
    {(r_1-s_1)^{1+\beta_1}(r_2-s_2)^{1+\beta_2}}\,dr\right).
\end{align*}

Therefore at the set $A_\eps$
\begin{align*}
&\norm{\Delta B^\eps}_{1,\beta_1,\beta_2}
\le\sup_{0\le s<t\le T}
    \left(c_\kappa\eps^{\delta/2}(t_1-s_1)^{h_1}(t_2-s_2)^{h_2}
    \vphantom{\int_{[s,t]}}\right.\\
&\quad+c_\kappa\eps^{\delta/2}(t_2-s_2)^{h_2}\int_{s_1}^{t_1}(u-s_1)^{h_1-1}\,du\\
&\quad+c_\kappa\eps^{\delta/2}(t_1-s_1)^{h_1}\int_{s_2}^{t_2}(v-s_2)^{h_2-1}\,dv\\
&\quad\left.+c_\kappa\eps^{\delta/2}\int_{[s,t]}(r_1-s_1)^{h_1-1}(r_2-s_2)^{h_2-1}\,dr\right)\\
&=\sup_{0\le s<t\le T}\left(c_\kappa\eps^{\delta/2}(t_1-s_1)^{h_1}(t_2-s_2)^{h_2}
\left(1+\frac{1}{h_1}+\frac{1}{h_2}+\frac{1}{h_1h_2}\right)\right)\\
&\le c_\kappa\eps^{\delta/2} T_1^{h_1}T_2^{h_2}
\left(1+\frac{1}{h_1}+\frac{1}{h_2}+\frac{1}{h_1h_2}\right)\to0,
\quad\eps\to0+.
\end{align*}

Then for any $a>0$
\[
\varlimsup_{\eps\to0+}\Prob\left(\norm{\Delta B^\eps}_{1,\beta_1,\beta_2}\ge a\right)\le \kappa,
\]
because for sufficiently small $\eps$ one has
$c_\kappa\eps^{\delta/2} T_1^{h_1}T_2^{h_2}(1+\frac{1}{h_1}+\frac{1}{h_2}+\frac{1}{h_1h_2})<a$.

Hence, when
$\kappa\to0+$
we have
\[
\lim_{\eps\to0+}\Prob\left(\norm{\Delta B^\eps}_{1,\beta_1,\beta_2}\ge a\right)=0.\qedhere
\]
\end{proof}

\begin{remark}
The convergence in probability in the last theorem may be not sufficient for some applications. For example, one may want to use this theorem to get an approximate solution to stochastic PDE with a fractional noise by solving a usual PDE with random force. She is not able of course, to solve for all $\omega$'s and takes some fixed $\omega$. She knows, of course, that there is a subsequence of solutions converging almost surely, but apriori it is not known which subsequence is it. So another subsequence (depending, say, on $\omega$) may converge to something different or there might be several such subsequences. To overcome this problem, we give below a proof that for $\eps_n=2^{-n}$ one has an almost sure convergence, and give moreover an estimate for the rate of convergence.
\end{remark}
Let
\begin{align*}
a(\eps)&=\sup_{s,t\in[0,T]}\frac{\abs{\Delta_s(\Delta B_t^\eps)}^p}{(\abs{t_1-s_1}\abs{t_2-s_2})^{p(\lambda-2\theta-\delta)/2}},\\
b(\eps)&=\eps^{\delta p/2}\ln^{1+\gamma}\eps,\quad\gamma>0.
\end{align*}
\eqref{6} implies that
\[
\E a(\eps)\le C_3\eps^{\delta p/2}.
\]
Then
\[
\E\left[\sum_{n\ge1}\frac{a(\eps_n)}{b(\eps_n)}\right]
=\E\left[\sum_{n\ge1}\frac{a(2^{-n})}{b(2^{-n})}\right]
\le C\sum_{n\ge1}\frac1{n^{1+\gamma}}\to0,
\quad n\to\infty.
\]
Therefore,
\[
\frac{a(\eps_n)}{b(\eps_n)}\to0,
\quad n\to\infty,\quad\text{a.\,s.}
\]
Hence,
\[
a(\eps_n)\le C(\omega)b(\eps_n)\quad\text{a.\,s.}
\]
So for all $s,t\in[0,T]$ and $n\ge1$
\[
\abs{\Delta_s(\Delta B_t^{\eps_n})}
\le C^{1/p}(\omega)\eps_n^{\delta/2}\ln^{(1+\gamma)/p}\eps_n
(\abs{t_1-s_1}\abs{t_2-s_2})^{(\lambda-2\theta-\delta)/2}
\quad\text{a.\,s.}
\]
Using the definition of the norm
$\norm{\cdot}_{1,\beta_1,\beta_2}$,
we obtain
\begin{multline*}
\norm{\Delta B^\eps_n}_{1,\beta_1,\beta_2}
\le C^{1/p}(\omega)T_1^{h_1}T_2^{h_2}
\left(1+\frac{1}{h_1}+\frac{1}{h_2}+\frac{1}{h_1h_2}\right)\\
\times\eps_n^{\delta/2}\ln^{(1+\gamma)/p}\eps_n\to0,
\quad n\to\infty,\quad\text{a.\,s.}
\end{multline*}

\section{Examples}
\label{sec4}
\subsection{Fractional Brownian sheet}
Fractional Brownian fields in the plane can be defined in various ways.
We consider the so-called anisotropic random fields that possess the fractional Brownian property coordinate-wise
(see e.\,g. \cite[Section~1.20]{Mishura08}).
\begin{definition}
A random field $\set{B_t^H, t\in[0,T]}$ is called a
\emph{fractional Brownian sheet} with Hurst index
$H=(H_1,H_2)\in(0,1)^2$
if
\begin{enumerate}
    \item $B_t^H$ is a Gaussian field such that $B_t^H=0$,
    $t\in\partial\R_+^2$;
    \item $\E B_t^H=0$, $\E B_t^HB_s^H=\frac14\prod\limits_{i=1,2}
    \left(t_i^{2H_i}+s_i^{2H_i}-\abs{t_i-s_i}^{2H_i}\right)$.
\end{enumerate}
\end{definition}
This field has a continuous modification.
Its increments satisfy the equality
\[
\E\left(\Delta_sB_t^H\right)^2=\abs{t_1-s_1}^{2H_1}\abs{t_2-s_2}^{2H_2}.
\]
Hence, for $B_t^H$ the inequality~\eqref{0} holds with~$\lambda=2\min\{H_1,H_2\}$.
Therefore, according to Theorem~\ref{th2.1}, for $H_i\in(\frac12,1)$ for all $\beta_1,\beta_2\in(0,H_1\wedge H_2)$ one has a convergence of approximations
\begin{gather*}
\norm{B^{H,\eps}-B^H}_{1,\beta_1,\beta_2}\xrightarrow{\Prob}0, \quad \eps\to0+,
\intertext{where}
B_t^{H,\eps}=\frac{1}{\eps^2}\int_{t_1}^{t_1+\eps}\!\!\int_{t_2}^{t_2+\eps}B_s^H\,ds.
\end{gather*}

\subsection{Multifractional Brownian sheet}
We consider a function
\[
H(t)=(H_1(t),H_2(t))\colon[0,T]\to(1/2,1)^2.
\]
Let $\mu$, $\nu$ be constants such that
\[
\frac12<\mu<\min_{t\in[0,T]}H_i(t)\le\max_{t\in[0,T]}H_i(t)<\nu<1.
\]
Assume that there exist positive constants $c_1$, $c_2$ such that for all $t,s\in [0,T]$
\begin{enumerate}
\renewcommand\theenumi{\arabic{enumi}}
\renewcommand\labelenumi{(H\theenumi)}
\item\label{H1}
    $\abs{H_i(t)-H_i(s)}\le c_1\left(\abs{t_1-s_1}^\nu+\abs{t_2-s_2}^{\nu}\right)$,
\item\label{H2}
    $\abs{\Delta_sH_i(t)}\le c_2\left(\abs{t_1-s_1}\abs{t_2-s_2}\right)^{\nu}$.
\end{enumerate}
\begin{definition}
Multifractional Brownian sheet
$\set{B_t^{H(t)},t\in[0,T]}$
with Hurst function $H(t)$ is defined as
\[
B_t^{H(t)}:=\int_{\R^2}\prod_{i=1,2}\left[(t_i-u_i)_+^{H_i(t)-1/2}-(-u_i)_+^{H_i(t)-1/2}\right]dW_u,
\]
$t\in[0,T]$,
where $s_+=\max\set{s,0}$,
$W=\set{W_s,s\in\R^2}$
is a standard Wiener field.
\end{definition}
Denote $Y_t=B_t^{H(t)}$.
\begin{theorem}\label{th3.1}
The trajectories of $B_t^{H(t)}$ are continuous with probability one.
\end{theorem}
\begin{proof}
We prove that trajectories of $Y_t=B_t^{H(t)}$ are continuous with probability one on any rectangle $[a,b]\subset(0,T]$ such that $\norm{b-a}<\delta$.
Moreover, according to~\cite[Lemma~2.2]{MeerschaertWuXiao}, the constant $\delta>0$ can be chosen such small that for all $s,t\in[a,b]$
the following inequality holds
\[
\E\left(Y_t-Y_s\right)^2
\le C_1\sum_{l=1}^2\abs{t_l-s_l}^{2\mu}
\le 2C_1\norm{t-s}^{2\mu}.
\]
As $Y_t$ is a Gaussian field, then for any $\alpha>0$ there exists $C_2>0$ such that
\[
\E\left(Y_t-Y_s\right)^\alpha
\le C_2\norm{t-s}^{\alpha\mu}.
\]
Taking $\alpha>2/\mu$, we obtain that the field $Y_t$ is continuous on $[a,b]$ with probability one according to Kolmogorov theorem.
\end{proof}

\begin{theorem}\label{th3.2}
There exists a constant $C>0$ such that for all
$s,t\in[0,T]$
\begin{equation}\label{3.1}
\E(\Delta_sY_t)^2\le C(\abs{t_1-s_1}\abs{t_2-s_2})^{2\mu}
\end{equation}
\end{theorem}
\begin{proof}
Denote $s'=(s_1,t_2)$,  $t'=(t_1,s_2)$.
Then
\[
\Delta_sY_t
=Y_t-Y_{t'}+Y_s-Y_{s'}
=A_1+A_2+A_3+A_4,
\]
where
\begin{align*}
A_1&=B_t^{H(t)}-B_{t'}^{H(t)}+B_s^{H(t)}-B_{s'}^{H(t)},\\
A_2&=B_{s'}^{H(t)}-B_s^{H(t)}+B_s^{H(s')}-B_{s'}^{H(s')},\\
A_3&=B_s^{H(t)}-B_s^{H(t')}+B_s^{H(s)}-B_s^{H(s')},\\
A_4&=B_{t'}^{H(t)}-B_{t'}^{H(t')}-B_{s}^{H(t)}+B_{s}^{H(t')}.
\end{align*}
Hence,
\begin{equation}\label{3.2}
\E(\Delta_sY_t)^2
\le 4(\E A_1^2+\E A_2^2+\E A_3^2+\E A_4^2).
\end{equation}
We estimate each of 4 terms.

As the random field $B_t^{H(t)}$ is a fractional Brownian motion when $H(t)=\const$, then
\begin{equation}\label{A1}
\E A_1^2\le C_3\abs{t_1-s_1}^{2H_1(t)}\abs{t_2-s_2}^{2H_2(t)}
\le C_4(\abs{t_1-s_1}\abs{t_2-s_2})^{2\mu}.
\end{equation}

Consider $A_2$.
\[
A_2=A_{21}+A_{22},
\]
where
\begin{align*}
A_{21}&=B_{s'}^{H(t)}-B_s^{H(t)}+B_{s}^{(H_1(s'),H_2(t))}-B_{s'}^{(H_1(s'),H_2(t))},\\
A_{22}&=B_{s'}^{(H_1(s'),H_2(t))}-B_{s}^{(H_1(s'),H_2(t))}+B_s^{H(s')}-B_{s'}^{H(s')}.
\end{align*}
\begin{align*}
&\E A_{21}^2=
E\left(B_{s'}^{H(t)}-B_s^{H(t)}+B_{s}^{(H_1(s'),H_2(t))}-B_{s'}^{(H_1(s'),H_2(t))}\right)^2\\
&=\int_{\R^2}\left(\left[(s_1-u_1)_+^{H_1(t)-\frac12}-(-u_1)_+^{H_1(t)-\frac12}\right]
    \left[(t_2-u_2)_+^{H_2(t)-\frac12}-(-u_2)_+^{H_2(t)-\frac12}\right]\right.\\
&-\left[(s_1-u_1)_+^{H_1(t)-\frac12}-(-u_1)_+^{H_1(t)-\frac12}\right]
    \left[(s_2-u_2)_+^{H_2(t)-\frac12}-(-u_2)_+^{H_2(t)-\frac12}\right]\\
&+\left[(s_1-u_1)_+^{H_1(s')-\frac12}-(-u_1)_+^{H_1(s')-\frac12}\right]
    \left[(s_2-u_2)_+^{H_2(t)-\frac12}-(-u_2)_+^{H_2(t)-\frac12}\right]\\
&-\left.\left[(s_1-u_1)_+^{H_1(s')-\frac12}-(-u_1)_+^{H_1(s')-\frac12}\right]
    \left[(t_2-u_2)_+^{H_2(t)-\frac12}-(-u_2)_+^{H_2(t)-\frac12}\right]\right)^2du\\
&=\int_\R\left(\left[(s_1-u_1)_+^{H_1(t)-\frac12}-(-u_1)_+^{H_1(t)-\frac12}\right]
    -\left[(s_1-u_1)_+^{H_1(s')-\frac12}-(-u_1)_+^{H_1(s')-\frac12}\right]\right)^2du_1\\
&\times\int_\R\left(\left[(t_2-u_2)_+^{H_2(t)-\frac12}-(-u_2)_+^{H_2(t)-\frac12}\right]
    -\left[(s_2-u_2)_+^{H_2(t)-\frac12}-(-u_2)_+^{H_2(t)-\frac12}\right]\right)^2du_2
\end{align*}

By Lemma~\ref{lem_fbm},
\[
\E A_{21}^2\le K_2(H_1(t)-H_1(s'))^2\cdot K_1\abs{t_2-s_2}^{2H_2(t)}
\]
Taking into account Condition~(H\ref{H1}), we get
\begin{equation}\label{A21}
\E A_{21}^2\le C_5(\abs{t_1-s_1}\abs{t_2-s_2})^{2\mu}
\end{equation}

Now we estimate $A_{22}$.
\begin{align*}
&\E A_{22}^2=
E\left(B_{s'}^{(H_1(s'),H_2(t))}-B_{s}^{(H_1(s'),H_2(t))}+B_s^{H(s')}-B_{s'}^{H(s')}\right)^2\\
&=\int_\R\left[(s_1-u_1)_+^{H_1(s')-\frac12}-(-u_1)_+^{H_1(s')-\frac12}\right]^2du_1\\
&\times\int_\R\left(\left[(t_2-u_2)_+^{H_2(t)-\frac12}-(-u_2)_+^{H_2(t)-\frac12}\right]\right.\\
&-\left[(s_2-u_2)_+^{H_2(t)-\frac12}-(-u_2)_+^{H_2(t)-\frac12}\right]
    +\left[(s_2-u_2)_+^{H_2(s')-\frac12}-(-u_2)_+^{H_2(s')-\frac12}\right]\\
&-\left.\left[(t_2-u_2)_+^{H_2(s')-\frac12}-(-u_2)_+^{H_2(s')-\frac12}\right]\right)^2du_2
\end{align*}

By Lemma~\ref{lem_fbm},
\[
\E A_{22}^2\le K_1s_1^{2H_1(s')}K_3(t_2-s_2)^{2\mu}(H_2(t)-H_2(s'))^2
\]
Considering the Condition~(H\ref{H1}), we obtain
\begin{equation}\label{A22}
\E A_{22}^2\le C_6(\abs{t_1-s_1}\abs{t_2-s_2})^{2\mu}
\end{equation}
\eqref{A21} and~\eqref{A22} imply that
\begin{equation}\label{A2}
\E A_2^2\le C_7(\abs{t_1-s_1}\abs{t_2-s_2})^{2\mu}
\end{equation}

Further, by Lemma~\ref{lem_fbs},
\begin{align*}
\E A_3^2&\le L\left[(H_1(t)-H_1(t')+H_1(s)-H_1(s'))^2+(H_2(t)-H_2(t')\right.\\
&+H_2(s)-H_2(s'))^2+\left((H_1(t)-H_1(t'))^2+(H_2(t)-H_2(t'))^2\right.\\
&+\left.(H_1(s)-H_1(s'))^2+(H_2(s)-H_2(s'))^2\right)\left((H_1(t)-H_1(s'))^2\right.\\
&+\left.\left.(H_2(t)-H_2(s'))^2+(H_1(s)-H_1(t'))^2+(H_2(s)-H_2(t'))^2\right)\right],
\end{align*}
and taking into account the conditions~\eqref{H1} and~\eqref{H2}, we get
\begin{equation}\label{A3}
\E A_3^2\le C_8(\abs{t_1-s_1}\abs{t_2-s_2})^{2\mu}.
\end{equation}

The estimation of $A_4$ is analogous to that of $A_2$.
We have
\begin{equation}\label{A4}
\E A_4^2\le C_9(\abs{t_1-s_1}\abs{t_2-s_2})^{2\mu}.
\end{equation}

Combining~\eqref{3.2}, \eqref{A1}, \eqref{A2}--\eqref{A4}, we get~\eqref{3.1}.
\end{proof}

Theorems~\ref{th3.1} and~\ref{th3.2} imply that multifractional Brownian sheet
$B_t^{H(t)}$
satisfies the conditions of Theorem~\ref{th2.1}.
Therefore,
\begin{corollary}
For any
$\beta_1,\beta_2\in(0,\mu)$
\[
\norm{B_t^{H(t),\eps}-B_t^{H(t)}}_{1,\beta_1,\beta_2}\xrightarrow{\Prob}0, \quad \eps\to0+,
\]
where
\[
B_t^{H(t),\eps}=\frac{1}{\eps^2}\int_{t_1}^{t_1+\eps}\!\!\int_{t_2}^{t_2+\eps}B_s^{H(s)}\,ds.
\]
\end{corollary}

\section{Appendix}
In this section we prove some technical lemmas that have been used in the proof of Theorem~\ref{th3.2}.
\subsection{Bounds for fractional Brownian motion}
Let
$(\Omega, \F, \Prob)$
be a complete probability space.
\begin{definition}
\emph{Fractional Brownian motion} with Hurst index
$H\in(0,1)$ is a centered Gaussian process $\widehat Z^{H}=\{\widehat Z_{t}^{H}, t\ge 0\}$
with stationary increments and the covariance function
\[
\E\left(\widehat Z_{t}^{H}\widehat Z_{s}^{H}\right)
=\frac{1}{2}\left(t^{2H}+s^{2H}-|t-s|^{2H}\right).
\]
\end{definition}

It is well-known that fractional Brownian motion has a continuous modification and can be represented in the following form.
\[
\widehat Z_t^H
=C_H\int_\R\left[(t-u)_+^{H-\frac12}-(-u)_+^{H-\frac12}\right]dW_u,
\]
where $W$ is a Wiener process,
$C_H=\frac{(2H\,\sin\pi H\,\Gamma(2H))^{1/2}}{\Gamma(H+1/2)}$
(see~\cite[Chapter~1.3]{Mishura08}).

Let
$\frac12<\mu<H_{\min}\le H_{\max}<\nu<1$.
Consider a family of random variablles
\[
Z_t^H
=\int_\R\left[(t-u)_+^{H-\frac12}-(-u)_+^{H-\frac12}\right]dW_u
=C_H^{-1}\widehat Z_t^H,
\]
$t\in[0,T]$, $H\in[\mu,\nu]$.

\begin{lemma}\label{lem_fbm}
There exist positive constants $K_1$, $K_2$, $K_3$ such that
\begin{enumerate}
\item for all $t_1,t_2\in[0,T]$, $H\in[H_{\min},H_{\max}]$
\begin{equation}\label{fbm1}
    \E\left(Z_{t_1}^H-Z_{t_2}^H\right)^2\le K_1\abs{t_1-t_2}^{2H};
\end{equation}
\item for all $t\in[0,T]$, $H_1,H_2\in[H_{\min},H_{\max}]$
\begin{equation}\label{fbm2}
    \E\left(Z_t^{H_1}-Z_t^{H_2}\right)^2\le K_2(H_1-H_2)^2.
\end{equation}
\item for all $t_1,t_2\in[0,T]$, $H_1,H_2\in[H_{\min},H_{\max}]$
\begin{equation}\label{fbm3}
    \E\left(Z_{t_1}^{H_1}-Z_{t_2}^{H_1}-Z_{t_1}^{H_2}+Z_{t_2}^{H_2}\right)^2\le K_3(t_1-t_2)^{2\mu}(H_1-H_2)^2.
\end{equation}
\end{enumerate}
\end{lemma}
\begin{proof}
(i)
By the definition,
\[
\E\left(Z_{t_1}^H-Z_{t_2}^H\right)^2
=C_H^{-2}\E\left(\widehat Z_{t_1}^H-\widehat Z_{t_2}^H\right)^2
=C_H^{-2}\abs{t_1-t_2}^{2H},
\]
which entails \eqref{fbm1} because $C_H^{-2}$ is bounded when $H\in[\mu,\nu]$.

(ii)
The inequality~\eqref{fbm2} is a corollary of~\eqref{fbm3}.

(iii)
We prove~\eqref{fbm3}.
Let for definiteness,
$t_2\le t_1$, $H_2\le H_1$.
We can write
\begin{align*}
&\E\left(Z_{t_1}^{H_1}-Z_{t_2}^{H_1}-Z_{t_1}^{H_2}+Z_{t_2}^{H_2}\right)^2\\
&\quad=\int_\R\left[(t_1-u)_+^{H_1-\frac12}-(t_2-u)_+^{H_1-\frac12}-(t_1-u)_
    +^{H_2-\frac12}+(t_2-u)_+^{H_2-\frac12}\right]^2du\\
&\quad=I_1+I_2,
\end{align*}
where
\begin{align*}
I_1&=\int_{-\infty}^{t_2}\left[(t_1-u)^{H_1-\frac12}-(t_2-u)^{H_1-\frac12}
    -(t_1-u)^{H_2-\frac12}+(t_2-u)^{H_2-\frac12}\right]^2du,\\
I_2&=\int_{t_2}^{t_1}\left[(t_1-u)^{H_1-\frac12}-(t_1-u)^{H_2-\frac12}\right]^2du.
\end{align*}
By the theorem on finite increments, there exists
$h\in[H_2,H_1]$ such that
\begin{equation}\label{fbm6}
\begin{split}
I_1&=(H_1-H_2)^2\int_{-\infty}^{t_2}\left[(t_1-u)^{h-\frac12}\ln(t_1-u)
    -(t_2-u)^{h-\frac12}\ln(t_2-u)\right]^2du\\
&=(H_1-H_2)^2\int_{-\infty}^{t_2}\left[\int_{t_2}^{t_1}
    (v-u)^{h-\frac32}\left(1+\left(h-\frac12\right)\ln(v-u)\right)dv\right]^2du.
\end{split}
\end{equation}

First we prove that for ahy
$h\in[\mu,\nu]$
\begin{equation}\label{fbm7}
\int_{-\infty}^{t_2}\left[\int_{t_2}^{t_1}(v-u)^{h-3/2}dv\right]^2du
\le C_3(t_1-t_2)^{2\mu},
\end{equation}
where $C_3>0$ is a constant.
Actually
\begin{align*}
&\int_{-\infty}^{t_2}\left[\int_{t_2}^{t_1}(v-u)^{h-3/2}dv\right]^2du\\
&=\int_{-\infty}^{2t_2-t_1}\left[\int_{t_2}^{t_1}(v-u)^{h-3/2}dv\right]^2du
    +\int_{2t_2-t_1}^{t_2}\left[\int_{t_2}^{t_1}(v-u)^{h-3/2}dv\right]^2du\\
&\le\int_{-\infty}^{2t_2-t_1}\left[\int_{t_2}^{t_1}(t_2-u)^{h-3/2}dv\right]^2du
    +\int_{2t_2-t_1}^{t_2}\left[\int_{t_2}^{t_1}(v-t_2)^{h-3/2}dv\right]^2du\\
&=(t_1-t_2)^{2h}\left((2-2h)^{-1}+(h-1/2)^{-2}\right)\\
&=\left(\frac{t_1-t_2}{T+1}\right)^{2h}(T+1)^{2h}\left((2-2h)^{-1}+(h-1/2)^{-2}\right)\\
&\le (t_1-t_2)^{2\mu}(T+1)\left((2-2\nu)^{-1}+(\mu-1/2)^{-2}\right).
\end{align*}

It is obvious that for all
$\eps,\delta>0$
there exist positive constants
$C_\eps$ and $\wt C_\delta$ such that
\begin{align}
\abs{\ln x}&\le C_\eps x^{-\eps}&&\text{if }0<x\le1,\label{fbm8}\\
\abs{\ln x}&\le \wt C_\delta x^{\delta}&&\text{if }x\ge1.\label{fbm9}
\end{align}

First we show that
\begin{equation}\label{fbm10}
\int_{-\infty}^{t_2-1}\left[\int_{t_2}^{t_1}
    (v-u)^{h-3/2}\left(1+(h-1/2)\ln(v-u)\right)dv\right]^2du
\le C_4(t_1-t_2)^{2\mu}.
\end{equation}
Applying the inequality~\eqref{fbm9} with $\delta:=(\nu-H_{\max})/2$, we obtain
\begin{align*}
&\int_{-\infty}^{t_2-1}\left[\int_{t_2}^{t_1}
    (v-u)^{h-\frac32}\left(1+(h-1/2)\ln(v-u)\right)dv\right]^2du\\
&\quad\le 2\int_{-\infty}^{t_2-1}\left[\int_{t_2}^{t_1}(v-u)^{h-\frac32}dv\right]^2du\\
&\quad\quad+2(\nu-1/2)^2\wt C_\delta^2 \int_{-\infty}^{t_2-1}\left[\int_{t_2}^{t_1}(v-u)^{h-\frac32+\delta}dv\right]^2du.
\end{align*}
Using~\eqref{fbm7}, we get~\eqref{fbm10}.

Secondly we prove that
\begin{equation}\label{fbm11}
\int_{t_2-1}^{t_2}\left[\int_{t_2}^{t_1}
    (v-u)^{h-3/2}\left(1+(h-1/2)\ln(v-u)\right)dv\right]^2du
\le C_5(t_1-t_2)^{2\mu}.
\end{equation}
Choose $\eps:=(\mu-H_{\min})/2$.
Then~\eqref{fbm8} implies that if
$u\in[t_2-1,t_2]$,
$v\in[t_2,t_1]$
then
\[
\abs{\ln(v-u)}=\abs{\ln(T+1)+\ln\frac{v-u}{T+1}}
\le\ln(T+1)+C_\eps\left(\frac{v-u}{T+1}\right)^{-\eps}.
\]
Therefore
\begin{multline*}
\int_{t_2-1}^{t_2}\left[\int_{t_2}^{t_1}
    (v-u)^{h-3/2}\left(1+(h-1/2)\ln(v-u)\right)dv\right]^2du\\
\le 2(1+(\nu-1/2)\ln(T+1))^2
    \int_{t_2-1}^{t_2}\left[\int_{t_2}^{t_1}
    (v-u)^{h-3/2}dv\right]^2du\\
+2(\nu-1/2)^2C_\eps^2(T+1)^{2\eps}\int_{t_2-1}^{t_2}\left[\int_{t_2}^{t_1}
    (v-u)^{h-3/2-\eps}dv\right]^2du.
\end{multline*}
Considering~\eqref{fbm7}, we obtain~\eqref{fbm11}.

Combining~\eqref{fbm6}, \eqref{fbm10}, \eqref{fbm11}, we get
\[
I_1\le C_6(t_1-t_2)^{2\mu}(H_1-H_2)^2.
\]

It remains to estimate $I_2$.
By the theorem on finite increments, there exists
$h\in[H_2,H_1]$
such that
\[
I_2=(H_1-H_2)^2\int_{t_2}^{t_1}\left[(t_1-u)^{h-1/2}\ln(t_1-u)\right]^2du.
\]
Choose $\eps:=(\mu-H_{\min})/2$.
Then~\eqref{fbm8} implies that if
$u\in[t_2,t_1]$
then
\[
\abs{\ln(t_1-u)}
=\abs{\ln(T+1)+\ln\frac{t_1-u}{T+1}}
\le\ln(T+1)+C_\eps\left(\frac{t_1-u}{T+1}\right)^{-\eps},
\]
which entails
\begin{align*}
I_2&\le2(H_1-H_2)^2\left(\frac{\ln^2(T+1)}{2h}(t_1-t_2)^{2h}
    +\frac{C_\eps^2(T+1)^{2\eps}}{2(h-\eps)}(t_1-t_2)^{2(h-\eps)}\right)\\
&\le C_7(t_1-t_2)^{2\mu}(H_1-H_2)^2.
\end{align*}
Now the proof is complete.
\end{proof}

\subsection{Bounds for integrals}
Let
$\frac12<\mu<H_{\min}\le H_{\max}<\nu<1$.
\begin{lemma}\label{lem_int}
Let
\[
f(t,u,h)=(t-u)_+^{h-1/2}-(-u)_+^{h-1/2},
\quad t\in[0,T],u\in\R, h\in[H_{\min},H_{\max}].
\]
Then for all $t\in[0,T]$, $h\in[H_{\min},H_{\max}]$
\begin{gather}
\int_\R[f(t,u,h)]^2du<+\infty,\label{int1}\\
\int_\R\left[f'_h(t,u,h)\right]^2du<+\infty,\label{int2}\\
\int_\R\left[f''_{hh}(t,u,h)\right]^2du<+\infty.\label{int3}
\end{gather}
\end{lemma}

\begin{proof}
We prove~\eqref{int3}.
Inequalities~\eqref{int1} and~\eqref{int2} are proved in a similar way.
\begin{align*}
\int_\R&\left[f''_{hh}(t,u,h)\right]^2du
=\int_{-\infty}^0\left[(t-u)^{h-\frac12}\ln^2(t-u)-(-u)^{h-\frac12}\ln^2(-u)\right]^2du\\
&\quad+\int_0^t(t-u)^{2h-1}\ln^4(t-u)du\\
&=\int_{-\infty}^0\left[\int_0^t(v-u)^{h-\frac32}
\left(\left(h-\frac12\right)\ln^2(v-u)+2\ln(v-u)\right)dv\right]^2du\\
&\quad+\int_0^t(t-u)^{2h-1}\ln^4(t-u)du
=:I_1+I_2+I_3,
\end{align*}
where
\begin{align*}
I_1&=\int_{-\infty}^{-1}\left[\int_0^t(v-u)^{h-\frac32}
\left(\left(h-\frac12\right)\ln^2(v-u)+2\ln(v-u)\right)dv\right]^2du,\\
I_2&=\int_{-1}^0\left[\int_0^t(v-u)^{h-\frac32}
\left(\left(h-\frac12\right)\ln^2(v-u)+2\ln(v-u)\right)dv\right]^2du,\\
I_3&=\int_0^t(t-u)^{2h-1}\ln^4(t-u)du.
\end{align*}

We study each of three terms.

1. Applying the inequality~\eqref{fbm9}, we obtain
\begin{align*}
I_1&\le\int_{-\infty}^{-1}\left[\int_0^t(v-u)^{h-3/2}
\left(\left(h-\frac12\right)\wt C_\delta^2(v-u)^{2\delta}
+2\wt C_\delta(v-u)^\delta\right)dv\right]^2du\\
&\le2\left(h-\frac12\right)^2\wt C_\delta^4
\int_{-\infty}^{-1}\left[\int_0^t(v-u)^{h-3/2+4\delta}dv\right]^2du\\
&\quad+8\wt C_\delta^2
\int_{-\infty}^{-1}\left[\int_0^t(v-u)^{h-3/2+2\delta}dv\right]^2du
\le Ct^{2\mu}<\infty
\end{align*}
(the last estimate follows from the inequality~\eqref{fbm7}).

2. Consider $I_2$.
\begin{multline*}
I_2=\int_{-1}^0\left[\int_0^t(v-u)^{h-3/2}
\left(\left(h-\frac12\right)\ln^2\frac{v-u}{T+1}+2\ln\frac{v-u}{T+1}\right.\right.\\
+\left.\left.\left(h-\frac12\right)\ln^2(T+1)+2\ln(T+1)\right)dv\right]^2du.
\end{multline*}
Using inequality~\eqref{fbm8}, we get
\begin{multline*}
I_2\le\int_{-1}^0\left[\int_0^t(v-u)^{h-3/2}
\left(C_\eps^2\left(h-\frac12\right)(T+1)^{2\eps}(v-u)^{-2\eps}\right.\right.\\
+2C_\eps(T+1)^{\eps}(v-u)^{-\eps}
+\left(h-\frac12\right)\ln^2(T+1)\\
+\left.\left.2\ln(T+1)\vphantom{\frac12}\right)dv\right]^2du
\le Ct^{2\mu}<\infty,
\end{multline*}
where the last estimate follows from the inequality~\eqref{fbm7}.

3. Consider $I_3$.
\begin{align*}
I_3&=\int_0^tv^{2h-1}\ln^4vdv
\le\int_0^1v^{2h-1}\ln^4vdv+\int_1^{1\vee t}v^{2h-1}\ln^4vdv\\
&\le\int_0^1v^{2\mu-1}\ln^4vdv+\int_1^{1\vee t}v^{2\nu-1}\ln^4vdv
<\infty.
\end{align*}
Thus, inequality~\eqref{int3} holds.
\end{proof}

\subsection{Bounds for fractional Brownian sheet}
Let
$\frac12<\mu<H_{\min}\le H_{\max}<\nu<1$.
We consider a family of random variables
\begin{multline*}
B_t^{H,H'}:=\int_{\R^2}\left((t_1-u_1)_+^{H-1/2}-(-u_1)_+^{H-1/2}\right)\\
\times\left((t_2-u_2)_+^{H'-1/2}-(-u_2)_+^{H'-1/2}\right)dW_u,
\end{multline*}
$t\in[0,T]$, $H,H'\in[\mu,\nu]$, $i=1,2$,
where $W=\set{W_s,s\in\R^2}$ is a Wiener field.

\begin{lemma}\label{lem_fbs}
There exists a constant $L>0$ such that
for all $t\in[0,T]$, $H_i,H_i'\in[H_{\min},H_{\max}]$, $i=1,2,3,4$, the following inequality holds
\begin{multline}\label{fbs1}
\E\left(B_t^{H_1,H_1'}-B_t^{H_2,H_2'}+B_t^{H_3,H_3'}-B_t^{H_4,H_4'}\right)^2\\
\le L\left((H_1-H_2+H_3-H_4)^2
+(H_1'-H_2'+H_3'-H_4')^2\right.\\
+\left((H_1-H_2)^2+(H_1'-H_2')^2+(H_3-H_4)^2+(H_3'-H_4')^2\right)\\
\left.\times\left((H_1-H_4)^2+(H_1'-H_4')^2+(H_2-H_3)^2+(H_2'-H_3')^2\right)\right).
\end{multline}
\end{lemma}

\begin{proof}
Denote
\begin{align*}
f(t,u,h)&=(t-u)_+^{h-1/2}-(-u)_+^{h-1/2},\\
f_1(h)&=f(t_1,u_1,h),\quad
f_2(h)=f(t_2,u_2,h).
\end{align*}
Then
\begin{gather*}
\E\left(B_t^{H_1,H_1'}-B_t^{H_2,H_2'}+B_t^{H_3,H_3'}-B_t^{H_4,H_4'}\right)^2
=\int_{\R^2}F^2(t,u)\,du,\\
F(t,u)=f_1(H_1)f_2(H_1')-f_1(H_2)f_2(H_2')
    +f_1(H_3)f_2(H_3')-f_1(H_4)f_2(H_4').
\end{gather*}
Consider two cases.

\emph{Case 1:}
$(H_1-H_2)(H_3-H_4)\ge0$.

In this case
\begin{align*}
\abs{H_1-H_2}&\le\abs{H_1-H_2+H_3-H_4},\\
\abs{H_3-H_4}&\le\abs{H_1-H_2+H_3-H_4}.
\end{align*}
We have
\begin{align*}
F(t,u)&=F_1(t,u)+F_2(t,u)+F_3(t,u),
\intertext{where}
F_1(t,u)&=(f_1(H_1)-f_1(H_2))f_2(H_2'),\\
F_2(t,u)&=(f_1(H_3)-f_1(H_4))f_2(H_3'),\\
F_3(t,u)&=f_1(H_1)(f_2(H_1')-f_2(H_2'))+f_1(H_4)(f_2(H_3')-f_2(H_4')).
\end{align*}
By the mean value theorem, there exist
$h_1\in[H_1\wedge H_2,H_1\vee H_2]$
and $h_2\in[H_3\wedge H_4,H_3\vee H_4]$ such that
\begin{align*}
f_1(H_1)-f_1(H_2)&=f_1'(h_1)(H_1-H_2),\\
f_1(H_3)-f_1(H_4)&=f_1'(h_2)(H_3-H_4).
\end{align*}
Applying Lemma~\ref{lem_int}, we get
\begin{align*}
\int_{\R^2}F_1^2(t,u)du
&\le(H_1-H_2)^2\int_{\R^2}(f_1'(h_1))^2\,du\int_{\R^2}(f_2(H_2'))^2\,du,\\
&\le C_1(H_1-H_2+H_3-H_4)^2.
\end{align*}
In much the same way, we have
\[
\int_{\R^2}F_2^2(t,u)du
\le C_2(H_1-H_2+H_3-H_4)^2.
\]
Hence,
\begin{align*}
\int_{\R^2}F^2(t,u)\,du
&\le3\left(\int_{\R^2}F_1^2(t,u)\,du+\int_{\R^2}F_2^2(t,u)\,du+\int_{\R^2}F_3^2(t,u)\,du\right)\\
&\le C_3(H_1-H_2+H_3-H_4)^2+3\int_{\R^2}F_3^2(t,u)\,du. \end{align*}
Consider the latter term.
There are two possible cases.

\emph{Case 1a:}
$(H_1'-H_2')(H_3'-H_4')\ge0$.

In this case
\begin{align*}
\abs{H_1'-H_2'}&\le\abs{H_1'-H_2'+H_3'-H_4'},\\
\abs{H_3'-H_4'}&\le\abs{H_1'-H_2'+H_3'-H_4'}.
\end{align*}
By the mean value theorem, there exist
$h_3\in[H_1'\wedge H_2',H_1'\vee H_2']$
and $h_4\in[H_3'\wedge H_4',H_3'\vee H_4']$
such that
\begin{align*}
f_2(H_1')-f_2(H_2')&=f_2'(h_3)(H_1'-H_2'),\\
f_2(H_3')-f_2(H_4')&=f_2'(h_4)(H_3'-H_4').
\end{align*}
Using Lemma~\ref{lem_int}, we estimate
\begin{align*}
\int_{\R^2}F_3^2(t,u)du
&\le(H_1'-H_2')^2\int_{\R^2}(f_1(H_1))^2\,du\int_{\R^2}(f_2'(h_3))^2\,du\\
&\quad+(H_3'-H_4')^2\int_{\R^2}(f_1(H_4))^2\,du\int_{\R^2}(f_2'(h_4))^2\,du\\
&\le C_4(H_1'-H_2'+H_3'-H_4')^2.
\end{align*}

\emph{Case 1b:}
$(H_1'-H_2')(H_3'-H_4')<0$.

Without loss of generality, we assume that
$\abs{H_1'-H_2'}>\abs{H_3'-H_4'}$ and $H_1'<H_2'$
(hence, $H_4'<H_3'$).

Put
$\wt H_1 = H_2'+H_4'-H_3'$.
It is not hard to see that
$\wt H_1\in[H_1',H_2']$.

\begin{align*}
F_3(t,u)&=F_{31}(t,u)+F_{32}(t,u),
\intertext{where}
F_{31}(t,u)&=f_1(H_1)\left(f_2(H_1')-f_2\left(\wt H_1\right)\right),\\
F_{32}(t,u)&=f_1(H_1)\left(f_2\left(\wt H_1\right)-f_2(H_2')\right)+f_1(H_4)(f_2(H_3')-f_2(H_4')).
\end{align*}

By the mean value theorem, there exists $h_5\in\left[H_1',\wt H_1\right]$ such that
\[
F_{31}(t,u)=f_1(H_1)f_2'(h_5)\left(H_1'-\wt H_1\right)
=f_1(H_1)f_2'(h_5)(H_1'-H_2'+H_3'-H_4').
\]
By Lemma~\ref{lem_int},
\begin{align*}
\int_{\R^2}F_{31}^2(t,u)du
&\le(H_1'-H_2'+H_3'-H_4')^2\int_{\R^2}(f_1(H_1))^2\,du\int_{\R^2}(f_2'(h_5))^2\,du\\
&\le C_5(H_1'-H_2'+H_3'-H_4')^2.
\end{align*}

We estimate $F_{32}(t,u)$.
\begin{align*}
F_{32}(t,u)&=-\int_0^{H_3'-H_4'}f_1(H_1)f_2'\left(\wt H_1+x\right)\,dx\\
&\quad+\int_0^{H_3'-H_4'}f_1(H_4)f_2'(H_4'+x)\,dx\\
&\le\int_0^{H_3'-H_4'}\int_{H_1}^{H_4}f_1'(y)f_2'\left(\wt H_1+x\right)\,dy\,dx\\
&\quad+\int_0^{H_3'-H_4'}\int_{\wt H_1}^{H_4'}
    f_1'(H_4)f_2''(x+y)\,dy\,dx.
\end{align*}
By the mean value theorem, there exist
$x_1,x_2\in[0,H_3'-H_4']$,
$y_1\in[H_1\wedge H_4,H_1\vee H_4]$,
$y_2\in\left[\wt H_1\wedge H_4',\wt H_1\vee H_4'\right]$
such that
\begin{align*}
\int_0^{H_3'-H_4'}&\int_{H_1}^{H_4}f_1'(y)f_2'\left(\wt H_1+x\right)\,dy\,dx\\
&=(H_1-H_4)(H_3'-H_4')f_1'(y_1)f_2'\left(\wt H_1+x_1\right),\\
\int_0^{H_3'-H_4'}&\int_{\wt H_1}^{H_4'}
    f_1'(H_4)f_2''(x+y)\,dy\,dx\\
&=\left(H_4'-\wt H_1\right)(H_3'-H_4')f_1'(H_4)f_2''(x_2+y_2)\\
&=(H_3'-H_2')(H_3'-H_4')f_1'(H_4)f_2''(x_2+y_2).
\end{align*}
Using Lemma~\ref{lem_int}, we get
\[
\int_{\R^2}F_{32}^2(t,u)du
\le C_6(H_3'-H_4')^2(H_1-H_4)^2+C_7(H_3'-H_4')^2(H_3'-H_2')^2.
\]
Thus, in the case~1b \eqref{fbs1} holds.

\emph{Case 2:}
$(H_1-H_2)(H_3-H_4)<0$.

Without loss of generality, we assume that
$\abs{H_1-H_2}>\abs{H_3-H_4}$ and $H_1<H_2$
(hence, $H_4<H_3$).

Put
$\wh H_1 = H_2+H_4-H_3$.
It is not hard to see that
$\wh H_1\in[H_1,H_2]$.

We have
\begin{align*}
F(t,u)&=G_1(t,u)+G_2(t,u)+G_3(t,u),
\intertext{where}
G_1(t,u)&=(f_1(H_1)-f_1(\wh H_1))f_2(H_1'),\\
G_2(t,u)&=(f_1(\wh H_1)-f_1(H_2))f_2(H_1')
    +(f_1(H_3)-f_1(H_4))f_2(H_4'),\\
G_3(t,u)&=f_1(H_2)(f_2(H_1')-f_2(H_2'))+f_1(H_3)(f_2(H_3')-f_2(H_4')).
\end{align*}

Terms $G_1(t,u)$ and $G_2(t,u)$ are estimated similarly to
$F_{31}(t,u)$ and $F_{32}(t,u)$ in Case 1b.
We have
\begin{align*}
\int_{\R^2}G_1^2(t,u)du
&\le C_8(H_1-H_2+H_3-H_4)^2,\\
\int_{\R^2}G_2^2(t,u)du
&\le C_9(H_3-H_4)^2(H_1'-H_4')^2+C_{10}(H_3-H_4)^2(H_3-H_2)^2.
\end{align*}
It remains to estimate $G_3(t,u)$.
We consider two cases.

\emph{Case 2a:}
$(H_1'-H_2')(H_3'-H_4')\ge0$.

$G_3(t,u)$ can be estimated similarly to $F_3(t,u)$ in Case~1a.
We get
\[
\int_{\R^2}G_3^2(t,u)du
\le C_{11}(H_1'-H_2'+H_3'-H_4')^2.
\]

\emph{Case 2b:}
$(H_1'-H_2')(H_3'-H_4')<0$.

This case can be considered in a similar way to Case~1b.
Without loss of generality, we assume that
$\abs{H_1'-H_2'}>\abs{H_3'-H_4'}$ and $H_1'<H_2'$.
Then we obtain
\begin{multline*}
\int_{\R^2}G_3^2(t,u)du
\le C_{12}(H_1'-H_2'+H_3'-H_4')^2\\
+C_{13}(H_3'-H_4')^2(H_3-H_2)^2+C_{14}(H_3'-H_4')^2(H_3'-H_2')^2.
\end{multline*}

Thus, now the proof is complete.
\end{proof}

\end{document}